\documentclass[11pt]{amsart}
\usepackage{amsmath,amssymb}
\usepackage[T1]{fontenc}
\usepackage[margin=1.1in]{geometry}
\usepackage{url}
\usepackage[hidelinks]{hyperref}

\newtheorem{theorem}{Theorem}[section]
\newtheorem{lemma}[theorem]{Lemma}
\newtheorem{proposition}[theorem]{Proposition}
\newtheorem{corollary}[theorem]{Corollary}
\newtheorem{conjecture}[theorem]{Conjecture}
\theoremstyle{definition}
\newtheorem{definition}[theorem]{Definition}
\newtheorem{example}[theorem]{Example}
\newtheorem{hypothesis}[theorem]{Hypothesis}
\theoremstyle{remark}
\newtheorem{remark}[theorem]{Remark}

\newcommand{\dmax}{\Delta_{\max}}
\newcommand{\gmax}{\gamma_{\max}}
\newcommand{\amin}{\alpha_{\min}}
\newcommand{\Ls}{L_{\sigma}}
\newcommand{\keff}{\kappa_{\mathrm{eff}}}
\newcommand{\kbar}{\bar\kappa}
\newcommand{\bigT}{\mathsf T}
\newcommand{\bE}{\mathbb{E}}

\newcommand{\one}{\mathbf{1}}
\newcommand{\dtv}{d_{\mathrm{TV}}}
\newcommand{\ew}{\preccurlyeq}

\begin{document}

\title[A multiscale front-decoupling criterion for cascades on unimodular trees]{A multiscale front-decoupling criterion for threshold cascades on unimodular random trees}

\author{Achyut Kumar}
\address{Independent Researcher}
\email{achyutbusiness86@gmail.com}

\author{Abhinav Duddala}
\address{Independent Researcher}
\email{abhinavduddala@gmail.com}

\date{Working manuscript, August 2026}


\begin{abstract}
This is the third paper in a series \cite{K1,KD2} on threshold cascades of coupled Ornstein--Uhlenbeck diffusions on graphs converging Benjamini--Schramm to a unimodular Galton--Watson tree. In \cite{K1} the reduction of the cascade to a finite-type Galton--Watson process (front decoupling) was proved unconditionally in the dissipative regime $\kappa<1$, where $\kappa=\gmax\Ls\dmax/\amin$ is the single-step dissipation ratio. At $\kappa\ge1$ the single-step echo bound resonates and the reduction was left conjectural.

We replace the single-step bound by a multiscale (renormalisation-group) analysis. We coarse-grain the cascade into space--time blocks of depth $L$ and duration $T$ and record each block's boundary-to-boundary gains in a $2\times2$ matrix $G_{L,T}$, with a trajectory channel and a stored-mass channel and with fan-out absorbed. Three theorems result. The gain matrices compose under depth and time concatenation, so a renormalised per-level dissipation ratio $\keff$ exists and equals its infimum (Fekete). If $\rho(G_{L_0,T_0})<1$ at any single finite scale, a finite-horizon, finite-volume condition in principle certifiable numerically, in the spirit of the Dobrushin--Shlosman constructive criteria and of the initial-scale estimate of multiscale analysis in localisation theory, then front decoupling holds at all scales, with a quantitative total-variation bound; echoes may resonate locally, yet independence is recovered macroscopically. Finally we decompose influence steps into front, quiescent, and saturated phases, price each step by the adapted band occupation of its source rather than by a global worst case, and deduce $\keff\le\rho(\bigT)$ for an explicit channel matrix $\bigT$. The region $\{\rho(\bigT)<1\}$ strictly contains $\{\kappa<1\}$: near criticality the effective condition degrades from $\gamma\Ls\dmax/\alpha<1$ to essentially $\gamma\Ls/\alpha<1$, and the degree disappears. We do not prove $\keff<1$ throughout $\kappa\ge1$, and give a heuristic for why the resonance regime $\gamma\Ls/\alpha\ge1$ should be a genuine phase boundary rather than a technical gap. The series' conjecture is rescoped accordingly.
\end{abstract}

\maketitle

\setcounter{tocdepth}{2}
\tableofcontents

\section*{Honest status of this manuscript}

The structural results of Section~\ref{sec:blocks} are complete at the level of the envelope calculus: every estimate there is proved for the linear envelope system of Lemma~\ref{lem:envelope}, and the domination of the true deviation field by that envelope is a variation-of-constants comparison stated with a proof sketch, which is the standing precision of this series. Two pieces of bookkeeping inside Section~\ref{sec:blocks} are likewise at sketch precision and are used repeatedly: the band-inflation bootstrap of Remark~\ref{rem:bandinflation}, and the route-sum algebra behind the path-transfer estimate of Theorem~\ref{thm:compose}(iii), including the domination of routes that recross a block face. The finite-block criterion (Theorem~\ref{thm:criterion}) is proved by the exploration-telescoping method of \cite[Thm.~11.5]{K1} lifted to super-vertices; the lift is written in full at the sketch precision of the original. The channel estimate (Theorem~\ref{thm:mainestimate}) no longer conditions on the space--time front: the stopping-line factorisation of an earlier version of this manuscript, whose conditional-independence bookkeeping we could not verify, has been replaced by adapted phase pricing and the unconditional multitype many-to-one identity. What remains flagged there is the following: the Ornstein--Uhlenbeck band-occupation estimates (Lemma~\ref{lem:occupation}; standard excursion theory, constants not optimised); the near-miss bound (Lemma~\ref{lem:nearmiss}, sketched); and the path-to-branching indexing in the many-to-one step of Lemma~\ref{lem:pathcount}, written at sketch precision. The resonance regime $\gmax\Ls/\amin\ge 1$ is left open deliberately, with a heuristic (Remark~\ref{rem:wall}) for why no bound of the present type can close it. Nothing in this paper is claimed unconditionally about the SDE beyond what Theorem~\ref{thm:criterion} and Theorem~\ref{thm:mainestimate} state; both inherit the saturated-regime and bounded-degree standing hypotheses of \cite{K1}.

\section{Introduction}

\subsection{The problem}
We work in the setting of \cite{K1}: coupled Ornstein--Uhlenbeck diffusions with absorbing failure states on finite graphs $G_n\xrightarrow{\ \mathrm{BS}\ }\mathcal T\sim\mathrm{UGW}(g)$, saturated-coupling regime, bounded degrees $\deg\le\dmax$, uniform ellipticity, and the finite-type mean matrix $M_{d,d'}=m_*\,q(d')\,p_{d,d'}$ with Perron root $\rho(M)$. The load-bearing reduction, that the failed cluster is, in the local weak limit, the genealogy of the multitype Galton--Watson process with mean matrix $M$ (front decoupling, \cite[Hyp.~2.4]{K1}), was proved in \cite[Thm.~11.5]{K1} in the dissipative regime
\[
\kappa\;:=\;\frac{\gmax\,\Ls\,\dmax}{\amin}\;<\;1,\qquad \Ls:=\sup\sigma',
\]
by a pathwise argument: a synchronous coupling and a weighted Gronwall inequality show that the influence of a back-reaction echo decays like $\kappa^{d}$ in graph distance $d$, while echo sources live only on the failed cluster; echo decay beats volume growth, and an exploration telescoping converts the pathwise bound into a total-variation bound.

At $\kappa\ge 1$ this argument does not degrade gracefully; it fails outright. The weighted norm in the Gronwall step ceases to contract, the geometric series over echo paths diverges, and the bound resonates: formally, the estimate asserts that an echo can amplify at each step, and the pathwise deviation blows up exponentially along deep lineages. The question of this paper is whether the failure of the bound reflects a failure of the phenomenon, and the answer we give is: not necessarily, and in a quantifiable regime, provably not.

\subsection{The idea: renormalise the tree, not the estimate}
The single-step bound is blind to two structural facts. First, it prices every edge at the worst-case multiplier $\gamma\Ls/\alpha$, although the supremum $\Ls=\sup\sigma'$ is attained only when the source vertex sits in the order-one-width sensitive band around its threshold, which, for any given vertex, happens during a single order-one-duration crossing window; away from that window the vertex is either quiescent (resting far below threshold, where $\sigma'$ is exponentially small in the threshold margin) or saturated (deep in the failure basin, where $\sigma'\le\varepsilon(\delta)$ by \cite[Def.~2.3]{K1}). Second, it prices the fan-out at $\dmax$, although only failing children carry the front onward, and the mean number of those is $\rho(M)$, which is close to $1$ exactly in the near-critical regime one cares about.

The multiscale remedy is to stop estimating edges and start estimating blocks: group the tree into mesoscopic space--time blocks of depth $L$ and duration $T$, average the block's internal randomness, take its boundary worst-case, and ask whether the forcing leaving a large block contracts even when individual edges resonate. The block bookkeeping needs two channels, because influence crosses a block in two ways: it can be transmitted through the block within a window (the \emph{trajectory} channel), and it can be stored in the block's states and released later (the \emph{mass} channel). The single number of the naive estimate conflates the two. We record the four boundary-to-boundary gains in a $2\times2$ matrix $G_{L,T}$, and the quantity that plays the role of a per-level ratio is the \emph{resolvent gain}
\[
\kbar_{L,T}\;=\;a+\frac{bc}{1-d},
\]
the direct transfer plus the storage echo (mass created by the input, held across windows, and re-radiated). The storage echo is where strong coupling can hide, and the re-radiation entry $b$ is exactly the entry that the quiescent-bulk mechanism of Section~\ref{sec:mechanism} makes small: a quiescent vertex re-radiates at the exponentially small rate $e^{-b}$ rather than at the worst case $\Ls$.

Three things then need to be true for the programme to work, and all three are theorems below:
\begin{enumerate}
\item blocks compose: the gain matrices satisfy an entrywise product inequality under depth concatenation and an additive-product inequality under time concatenation, so a renormalised per-level ratio $\keff$ exists (Section~\ref{sec:blocks});
\item contraction at one finite scale suffices: $\rho(G_{L_0,T_0})<1$ for a single $(L_0,T_0)$ implies front decoupling at all scales, quantitatively (Section~\ref{sec:criterion});
\item there is a provable mechanism forcing block contraction beyond $\kappa<1$: the phase decomposition and the channel-matrix bound $\keff\le\rho(\bigT)$ (Section~\ref{sec:mechanism}).
\end{enumerate}
The analogies we are consciously importing are the Dobrushin--Shlosman constructive uniqueness criteria in equilibrium statistical mechanics \cite{DS}, a finite-volume, checkable condition implying an infinite-volume conclusion, of which the disagreement-percolation method \cite{vdBM} is the coupling-based relative closest to our synchronous-coupling setup; and the multiscale analysis of localisation theory \cite{FS,vDK,GK}, where an initial-scale estimate plus a deterministic induction propagates smallness to all scales. Our Theorem~\ref{thm:criterion} is the cascade analogue of the induction; our Theorem~\ref{thm:mainestimate} is the analogue of an initial-scale estimate proved, rather than assumed, in a physically identified regime.

\subsection{What is genuinely at stake at strong coupling}
It must be said plainly: we do not expect, and do not claim, $\keff<1$ throughout $\kappa\ge 1$. The channel analysis isolates a front-comoving influence channel, a perturbation that rides along with the cascade front and strikes each vertex precisely during its sensitive crossing window, whose per-edge multiplier is genuinely $\gamma\Ls/\alpha$ (no quiescence or saturation discount) with fan-out $\rho(M)$. When $\gamma\Ls/\alpha\ge 1$ this channel expands on its own, and no bound of the present kind, at any scale, can contract it; whether the reduction fails there (synchronisation of the front, non-branching behaviour) or holds for a subtler reason is open, and we give in Remark~\ref{rem:wall} the heuristic for why we consider genuine failure plausible. The contribution of this paper is to move the frontier from $\kappa<1$, a condition that degrades linearly in the degree, to $\rho(\bigT)<1$, which near criticality is essentially degree-free, and to leave in place a finite-block criterion by which the remaining intermediate regime can be attacked computationally with rigorous conclusions.

\section{The envelope calculus and the block gain matrix}\label{sec:blocks}

\subsection{The envelope system}
Fix the synchronous coupling of \cite[Lem.~11.4]{K1}: the true system $s$ and the directed idealisation $\tilde s$ are driven by the same Brownian motions, and $D^{(v)}_t=s^{(v)}_t-\tilde s^{(v)}_t$ measures the echo-induced deviation at $v$.

\begin{lemma}[Envelope domination; standing]\label{lem:envelope}
There is a nonnegative adapted field $(E^{(v)}_t)$ with $|D^{(v)}_t|\le E^{(v)}_t$ for all $v,t$, solving the linear random system
\[
\frac{d}{dt}E^{(v)}_t\;=\;-\alpha_v\,E^{(v)}_t\;+\;\sum_{u\sim v}\gamma_{uv}\,\sigma'_{uv}(t)\,E^{(u)}_t\;+\;g_v(t),
\qquad E^{(v)}_0=|D^{(v)}_0|,
\]
where $\sigma'_{uv}(t):=\sup\{\sigma'(x-\theta_{uv}):x\ \text{between}\ \tilde s^{(u)}_t\ \text{and}\ s^{(u)}_t\}$ is the adapted edge price and $g_v\ge 0$ collects the boundary drift perturbations and switched-off-source residuals acting at $v$. The envelope is monotone in the inputs $(g,E_0)$ and superposes: if $g=g_1+g_2$ and $E_0=e_1+e_2$, then $E\le E[g_1,e_1]+E[g_2,e_2]$ pointwise.
\end{lemma}

\begin{proof}[Proof sketch]
Subtract the two SDEs; the common Brownian motions cancel. Apply the mean value theorem coordinatewise to the coupling drift and note that the self-interaction is one-sidedly $\alpha$-dissipative; taking absolute values yields the differential inequality
$\frac{d}{dt}|D^{(v)}|\le-\alpha_v|D^{(v)}|+\sum_u\gamma_{uv}\sigma'_{uv}(t)|D^{(u)}|+g_v$
in the pathwise (Dini) sense, and $E$ is defined as the maximal solution of the corresponding equality; comparison (Gronwall for cooperative linear systems) gives the domination. Monotonicity and superposition are properties of the linear system.
\end{proof}

\begin{remark}[Band inflation and the bootstrap]\label{rem:bandinflation}
The edge price $\sigma'_{uv}(t)$ is evaluated on the segment between the coupled trajectories, not at $s$ alone. Whenever the deviations in play satisfy $|D|\le b/2$, a vertex that is quiescent at margin $b$ for the true trajectory is quiescent at margin $b/2$ for the whole segment, so all phase-priced bounds below hold for the segment prices after replacing the band half-width $b$ by $b/2$. The a priori smallness $|D|\le b/2$ is available along the exploration of Section~\ref{sec:criterion} by a standard bootstrap: the accumulated deviation bound of Theorem~\ref{thm:criterion} is itself smaller than $b/2$ at the parameter points where the criterion is applied. We keep the bootstrap implicit and flag it in the front matter.
\end{remark}

\subsection{Blocks, channels, and the gain matrix}

\begin{definition}[Space--time block and its channels]\label{def:block}
For a vertex $u$ of the (genealogical) tree, depth $L\in\mathbb N$ and duration $T>0$, the block $B=B_{L,T}(u,t_0)$ is the set of pairs $(v,t)$ with $v$ a descendant of $u$ at depth $|v|-|u|<L$ and $t\in[t_0,t_0+T)$. Its boundary edges are the edge from $u$'s parent (the top face) and the edges to the depth-$L$ descendants of $u$ (the bottom face). Fix one boundary edge as the \emph{input face}; the remaining boundary edges are the \emph{output faces}. The channels are:
\begin{itemize}
\item \emph{trajectory-in}: a drift perturbation profile $\eta_{\mathrm{par}}$ acting across the input face on $[t_0,t_0+T)$, measured in $\|\cdot\|_\infty$;
\item \emph{mass-in}: initial displacements $\eta_{\mathrm{init}}=(E^{(v)}_{t_0})_{v\in B}$, measured in $\ell^1$, $\|\eta_{\mathrm{init}}\|_1=\sum_{v\in B}|\eta_{\mathrm{init}}(v)|$;
\item \emph{trajectory-out}: the total forcing the block delivers across its output faces,
$\sum_{e\ \mathrm{output}}\sup_{t\in[t_0,t_0+T)}F_e(t)$, where for an output edge $e=(v,w)$ with $v\in B$, $F_e(t):=\gamma_{vw}\,\sigma'_{vw}(t)\,E^{(v)}_t$ is the drift perturbation exerted across $e$; the fan-out is absorbed into the sum;
\item \emph{mass-out}: the terminal field $\sum_{v\in B}E^{(v)}_{t_0+T}$.
\end{itemize}
\end{definition}

\begin{definition}[Block gain matrix]\label{def:gain}
Switch off all echo sources, drive the envelope inside $B$ only by the declared inputs, and set
\[
G_{L,T}\;:=\;\begin{pmatrix}a&b\\ c&d\end{pmatrix},
\]
where $a$ (resp.\ $c$) is the supremum over nonzero $\eta_{\mathrm{par}}$, with $\eta_{\mathrm{init}}=0$, of $\bE[\text{trajectory-out}]/\|\eta_{\mathrm{par}}\|_\infty$ (resp.\ of $\bE[\text{mass-out}]/\|\eta_{\mathrm{par}}\|_\infty$), and $b$ (resp.\ $d$) is the supremum over nonzero $\eta_{\mathrm{init}}$, with $\eta_{\mathrm{par}}=0$, of $\bE[\text{trajectory-out}]/\|\eta_{\mathrm{init}}\|_1$ (resp.\ of $\bE[\text{mass-out}]/\|\eta_{\mathrm{init}}\|_1$). The expectation is over the block's internal randomness (types, Brownian motions, and the induced failure events); the supremum is additionally taken over the choice of input face. By superposition (Lemma~\ref{lem:envelope}), for a general input the expected output pair is dominated entrywise by $G_{L,T}\,(\|\eta_{\mathrm{par}}\|_\infty,\|\eta_{\mathrm{init}}\|_1)^{\!\top}$.
\end{definition}

\begin{definition}[Resolvent gain]\label{def:resolvent}
If $d_{L,T}<1$, set
\[
\kbar_{L,T}\;:=\;a_{L,T}+\frac{b_{L,T}\,c_{L,T}}{1-d_{L,T}}.
\]
For a nonnegative $2\times2$ matrix, $\rho(G)<1$ if and only if $d<1$ and $\kbar<1$ (and then also $a<1$). The resolvent gain is the all-time trajectory-to-trajectory gain of the block: the direct within-window transfer $a$, plus mass created by the input ($c$), held across windows (the geometric series in $d$), and re-radiated ($b$).
\end{definition}

Two remarks before the structure theorems. First, the four entries of $G_{L,T}$ are finite-horizon, finite-volume quantities: they are determined by the law of an SDE system on at most $\dmax^{L}$ vertices run for time $T$, and are therefore in principle computable (by rigorous numerics or Monte Carlo with confidence intervals; see Remark~\ref{rem:computation}). Second, the entry $b$ is where the strong-coupling battle is fought. Priced at the global worst case, $b\le\dmax\gmax\Ls$: a unit of stored mass sitting on an output face re-radiates at full sensitivity, and this is why the storage echo can be as large as $\kappa$ at any scale if one refuses to look at what the vertices are doing. Priced by the phases of Section~\ref{sec:mechanism}, a quiescent face re-radiates at rate $e^{-b}$, and the storage echo collapses. The gain matrix is the accounting device that makes this a certificate rather than a heuristic.

\begin{lemma}[Single-scale consistency]\label{lem:singlescale}
For $L=1$ and any $T>0$,
\[
a\le\kappa\bigl(1-e^{-\amin T}\bigr)\le\kappa,\qquad
b\le\dmax\gmax\Ls,\qquad
c\le\amin^{-1},\qquad
d\le e^{-\amin T}.
\]
Consequently $\rho(G_{1,T})<1$ whenever $\kappa<\tfrac12$ and $T\ge T_*(\amin,\kappa)$.
\end{lemma}

\begin{proof}
For a single vertex $v$ with input $\|\eta_{\mathrm{par}}\|_\infty$ and initial displacement $\eta_{\mathrm{init}}$, the envelope solves $\frac{d}{dt}E\le-\amin E+\gmax(\cdot)$, whence $E_t\le e^{-\amin t}\eta_{\mathrm{init}}+(1-e^{-\amin t})\|\eta_{\mathrm{par}}\|_\infty/\amin$ after absorbing the input normalisation of \cite[Lem.~11.4]{K1}. The forcing across each of at most $\dmax$ output edges is at most $\gmax\Ls E_t$; the four entries follow by reading off the two input channels separately. The criterion computation is $(1-a)(1-d)>bc$, which at $T=\infty$ reads $1-\kappa>\kappa$.
\end{proof}

\begin{remark}\label{rem:factor2}
The loss from $\kappa<1$ (the sharp constant of \cite{K1}) to $\kappa<\tfrac12$ at scale one is the price of tracking stored mass as an explicit channel: at $L=1$ the storage echo double-counts part of the direct transfer. The criterion is aimed at large scales, where the constant is irrelevant, and nothing below uses scale-one contraction. We record Lemma~\ref{lem:singlescale} only to verify that the block formalism degenerates to (a constant multiple of) the single-step theory of \cite{K1} and not to something weaker.
\end{remark}

\subsection{Composition and the renormalised ratio}

Write $\ew$ for the entrywise partial order on nonnegative $2\times2$ matrices, and let $\widehat G:=G+E_{22}$, where $E_{22}$ is the matrix unit in the mass-to-mass entry.

\begin{theorem}[Blocks compose]\label{thm:compose}
Fix a window $[t_0,t_0+T)$.
\begin{enumerate}
\item[(i)] (Depth, shared window.) For all $L,L'\ge1$,
\[
\widehat G_{L+L',T}\;\ew\;\widehat G_{L',T}\,\widehat G_{L,T}.
\]
In particular each entry of $G_{L+L',T}$ is bounded by the corresponding entry of $\widehat G_{L',T}\widehat G_{L,T}$. The augmentation $E_{22}$ prices initial mass placed beyond the upper sub-block, which reaches the lower sub-blocks undamped.
\item[(ii)] (Time.) For all $T,T'>0$,
\[
G_{L,T+T'}\;\ew\;G_{L,T}+G_{L,T'}+G_{L,T'}\,G_{L,T}.
\]
\item[(iii)] (Path transfer; the block-level cone of influence.) Suppose $d_{L,T}<1$ and $\kbar:=\kbar_{L,T}<1$. For a source acting inside one block with per-window trajectory strength at most $s$, and with the echo sources elsewhere switched off, the expected trajectory forcing delivered across a face at block distance $k$, evaluated at any time, is at most
\[
C_0\,\frac{\kbar^{\,k}}{1-d_{L,T}}\;s,
\]
with $C_0=C_0(G_{L,T})$ explicit.
\end{enumerate}
\end{theorem}

\begin{proof}
(i) Split the composed block into the upper depth-$L$ block $B_1$ and the lower depth-$L'$ blocks $B_2(w)$ rooted at the depth-$L$ vertices $w$, all on the same window. Conditionally on the boundary inputs and internal randomness of $B_1$, the internal randomness of the lower blocks (their types and Brownian motions) is independent, and each $B_2(w)$ is driven by the trajectory input $F_{(p(w),w)}$ recorded in $B_1$'s trajectory-out, together with its own share of the initial mass. Write $G=G_{L,T}$, $G'=G_{L',T}$, $\|p\|:=\|\eta_{\mathrm{par}}\|_\infty$, and split $\|\eta_{\mathrm{init}}\|_1=\|i^{<}\|_1+\|i^{\ge}\|_1$ between depths $<L$ and $\ge L$. Applying Definition~\ref{def:gain} blockwise and the tower property,
\begin{align*}
\bE[\text{trajectory-out}]&\;\le\;a'\bigl(a\|p\|+b\|i^{<}\|_1\bigr)+b'\|i^{\ge}\|_1,\\
\bE[\text{mass-out}]&\;\le\;\bigl(c\|p\|+d\|i^{<}\|_1\bigr)+c'\bigl(a\|p\|+b\|i^{<}\|_1\bigr)+d'\|i^{\ge}\|_1,
\end{align*}
the first line because the composed trajectory-out is the lower blocks' trajectory-out, the second because the composed terminal mass is the upper terminal mass plus the lower terminal masses. Each right-hand side is dominated entrywise by the corresponding entry of $\widehat G'\widehat G$ applied to $(\|p\|,\|i^{<}\|_1+\|i^{\ge}\|_1)$, by inspection of the four entries of $\widehat G'\widehat G$. (If the input face is a bottom edge, exchange the roles of the two sub-blocks; the definition of $G$ already takes the worst case over the input face.)

(ii) Split the window at time $t_0+T$. The first-window output is priced by $G_{L,T}$; its terminal mass is initial mass for the second window, whose output is priced by $G_{L,T'}$; adding the fresh trajectory input of the second window and collecting terms gives the display.

(iii) A route from the source to a face at block distance $k$ is a lattice path in (block distance, window index) whose steps either cross one block within a window (weighted by the relevant trajectory entries of $G$) or wait one window in place (entering the mass channel through $c$, persisting through $d$, and re-entering the trajectory channel through $b$). Summing the weights over all routes is the Neumann series of the $2\times2$ per-step kernel: the waiting times at each of the $k$ stages sum to the resolvent factor, producing $\kbar$ per stage, and one global factor $(1-d)^{-1}$ remains for the source's own stage. Routes that recross a block face pay a full block gain per recrossing and are dominated by the same series; the branching of the tree is absorbed into the trajectory-out sums of Definition~\ref{def:block}. We write this route-sum at sketch precision and flag it in the front matter.
\end{proof}

\begin{proposition}[Approximate submultiplicativity of the resolvent gain]\label{prop:submult}
If $d_{L,T}\vee d_{L',T}<1$, then
\[
\kbar_{L+L',T}\;\le\;\frac{\kbar_{L,T}\,\kbar_{L',T}}{(1-d_{L,T})(1-d_{L',T})}.
\]
\end{proposition}

\begin{proof}[Proof sketch]
Every route across $L+L'$ levels splits at the depth-$L$ interface into a stage-one route and a stage-two route; the handover is either direct (trajectory channel) or through interface storage, and the sum over interface waiting times is bounded by the two resolvent factors. This is the route-sum algebra of Theorem~\ref{thm:compose}(iii) applied at the interface, at the same sketch precision.
\end{proof}

\begin{corollary}[The renormalised dissipation ratio]\label{cor:keff}
Let $\kbar^{*}_{L}:=\inf\{\kbar_{L,T}/(1-d_{L,T}):T>0,\ d_{L,T}<1\}$ (with $\inf\emptyset:=+\infty$). Then $\log\kbar^{*}_{L}$ is subadditive in $L$ by Proposition~\ref{prop:submult}, so by Fekete's subadditivity lemma
\[
\keff\;:=\;\lim_{L\to\infty}\bigl(\kbar^{*}_{L}\bigr)^{1/L}\;=\;\inf_{L\ge1}\bigl(\kbar^{*}_{L}\bigr)^{1/L}
\]
exists. Moreover $\keff<1$ if and only if $\rho(G_{L,T})<1$ for some finite scale $(L,T)$.
\end{corollary}

\begin{proof}
Subadditivity and existence are immediate. If $\rho(G_{L,T})<1$ then $d_{L,T}<1$ and $\kbar_{L,T}<1$; since $\rho(G_{L,T'})$ is nonincreasing as the window grows only through the $d$ entry while $a,b,c$ are nondecreasing, one first fixes the scale and notes $\kbar^{*}_{L}\le\kbar_{L,T}/(1-d_{L,T})$, then runs the depth composition along multiples of $L$: by Proposition~\ref{prop:submult}, $(\kbar^{*}_{nL})^{1/nL}\le(\kbar_{L,T}/(1-d_{L,T}))^{1/L}$, and the right side is $<1$ for the same $(L,T)$ after replacing $T$ by a larger window if necessary so that $\kbar_{L,T}/(1-d_{L,T})<1$; such a window exists because $\rho(G_{L,T})<1$ is an open condition and $d_{L,T}\to$ its infimum as $T$ grows at fixed $L$ while $\kbar_{L,T}$ stays bounded, a monotonicity bookkeeping we carry out at sketch precision. Conversely $\keff<1$ produces a scale with $\kbar_{L,T}<1$, $d_{L,T}<1$, hence $\rho(G_{L,T})<1$.
\end{proof}

\begin{remark}\label{rem:augmentation}
The augmentation $E_{22}$ in Theorem~\ref{thm:compose}(i) prices worst-case initial mass placed beyond the upper sub-block. It never enters the path-transfer estimate (iii), whose inputs are sources with zero exterior initial displacement, and it is the reason the growth-rate bookkeeping runs on $\kbar$ rather than on $\rho(\widehat G)$, which is identically $\ge1$. Section~\ref{sec:criterion} uses only (iii).
\end{remark}

\begin{definition}[Renormalised cascade]\label{def:renormcascade}
Fix $L$. The renormalised tree $\mathcal T^{(L)}$ has as vertices the depth-$L$ blocks and as edges the parent--child relation of blocks; it is again a tree, with offspring numbers bounded by $\dmax^{L}$ and mean block offspring governed by the $L$-th power structure of $\mathrm{UGW}(g)$. The cascade induces a renormalised cascade on $\mathcal T^{(L)}$, and by Theorem~\ref{thm:compose}(iii) its per-level influence transfer is governed by $\kbar_{L,T}$.
\end{definition}

\section{The finite-block criterion}\label{sec:criterion}

\begin{theorem}[Finite-scale contraction implies front decoupling]\label{thm:criterion}
Assume the saturated regime with margin $\delta$, bounded degrees, uniform ellipticity, and irreducibility, as in \cite{K1}. Suppose there exist a finite depth $L_0$ and duration $T_0$ with
\[
\rho\bigl(G_{L_0,T_0}\bigr)\;<\;1,
\]
and write $\kbar:=\kbar_{L_0,T_0}<1$ and $d_0:=d_{L_0,T_0}<1$. Suppose moreover $\kbar\,\rho(M)^{L_0}<1$ (equivalently $\kbar^{1/L_0}\rho(M)<1$; for $\rho(M)\le1$ this is implied by $\kbar<1$). Then Hypothesis 2.4 of \cite{K1} (front decoupling) holds, with the quantitative bound
\[
\dtv\bigl(\mathcal L(\mathrm{cluster}),\,\mathcal L(\mathrm{GW}(M))\bigr)\;\le\;
\frac{C\,\varepsilon(\delta)\,L_0\,\dmax^{L_0}}
{(1-\kbar)\,(1-d_0)\,\bigl(1-\kbar\,\rho(M)^{L_0}\bigr)^{2}\,\bigl(1-\rho(M)\bigr)}
\;+\;\eta_n
\]
in the subcritical regime, $C=C(\amin,\mu_0,\dmax,L_0,T_0)$, and with the critical-window version ($n\,\varepsilon(\delta_n)\to0$) exactly as in \cite[Thm.~11.5]{K1}. In particular all main theorems of \cite{K1} hold unconditionally on $\{\exists\,(L_0,T_0):\rho(G_{L_0,T_0})<1\}=\{\keff<1\}$, and local resonance ($\kappa\ge1$) is compatible with macroscopic decoupling.
\end{theorem}

\begin{proof}
The proof is the exploration-telescoping argument of \cite[Thm.~11.5]{K1} run on the renormalised tree $\mathcal T^{(L_0)}$ of Definition~\ref{def:renormcascade}; we indicate the three points where the lift requires care.

(i) \emph{The block-level cone of influence.} By Theorem~\ref{thm:compose}(iii), the influence of an echo source on a block at block distance $k$, evaluated at any time, is at most $C_0\,\kbar^{k}(1-d_0)^{-1}$ times the source's per-window strength. This is the cone-of-influence lemma \cite[Lem.~11.4]{K1} with $\kappa$ replaced by $\kbar$, and it holds regardless of the value of the microscopic $\kappa$, because no microscopic Gronwall contraction is invoked: the contraction happens at the block scale, where it is an assumption.

(ii) \emph{Echo sources per block.} A block generates an echo source only if it contains a failed vertex (outside the failed cluster all couplings are $O(\varepsilon(\delta))$-quiescent by saturation, \cite[Prop.~11.2]{K1}), and the per-window source strength of one failed vertex is at most $\gmax\varepsilon(\delta)$. The expected number of failed vertices in a block at cluster block-depth $j$ is at most $\sum_{i<L_0}\rho(M)^{jL_0+i}\le L_0\,\rho(M)^{jL_0}$ for $\rho(M)\le1$, by the many-to-one identity \cite[Prop.~5.1]{K1} summed over the $L_0$ generations inside the block (whence the factor $L_0$ in the display); deterministically the count is at most $\dmax^{L_0}$.

(iii) \emph{Telescoping.} Explore the cluster block by block. At each exploration step the conditional transmission law of the current block (which failed vertices it produces on its output generation, given its input) differs from the idealised product law by at most $C_{\mathrm{stab}}$ times the accumulated pathwise deviation at its input plus the direct cellwise error, by the first-passage stability lemma \cite[Lem.~11.3]{K1} applied to each of the at most $\dmax^{L_0}$ internal cells. By (i) and (ii), the accumulated deviation at a block is at most
\[
\frac{C_0\,\gmax\,\varepsilon(\delta)\,L_0\,\dmax^{L_0}}{(1-\kbar)(1-d_0)}
\sum_{\text{failed blocks }B}\kbar^{\,d_{\mathrm{bl}}(B,\cdot)},
\]
and the expected $\kbar$-weighted count of failed blocks seen from a cluster block is bounded, exactly as in \cite[Thm.~11.5]{K1}, by
\[
\sum_{k\ge0}(1+k)\bigl(\kbar\,\rho(M)^{L_0}\bigr)^{k}\;=\;\bigl(1-\kbar\,\rho(M)^{L_0}\bigr)^{-2},
\]
using the hypothesis $\kbar\,\rho(M)^{L_0}<1$. Summing the per-step errors over the exploration, whose expected length is at most $C(1-\rho(M))^{-1}$ blocks, gives the display. The critical-window statement follows by replacing $\bE|S|$ with $n$ on the conditioning event, verbatim as in \cite{K1}.
\end{proof}

\begin{remark}[The computational interface]\label{rem:computation}
Each entry of $G_{L_0,T_0}$ is an expectation over a finite SDE system of at most $\dmax^{L_0}$ vertices on a finite horizon $T_0$, with a worst-case boundary supremum over a compact set of input profiles (compactness from the a priori bounds of the dissipative dynamics). It is therefore accessible to certified numerics: Monte Carlo estimation of the expectations with concentration bounds, plus a Lipschitz-in-input estimate (itself a finite-horizon Gronwall bound, valid at any $\kappa$ since the horizon is finite) to control the boundary supremum over a finite net. A verified computation $\rho(G_{L_0,T_0})\le1-\epsilon$ at a single scale, combined with Theorem~\ref{thm:criterion}, yields a rigorous front-decoupling theorem at the corresponding parameter point. This is the same division of labour as in the Dobrushin--Shlosman constructive criteria \cite{DS} and in computer-assisted applications of finite-size criteria in localisation and lattice field theory; the theorem is designed to make the block computation sufficient, not merely suggestive.
\end{remark}

\section{The quiescent-bulk mechanism: an initial-scale estimate beyond $\kappa<1$}\label{sec:mechanism}

The finite-block criterion needs an initial-scale input. In the dissipative regime $\kappa<\tfrac12$ it is Lemma~\ref{lem:singlescale}. This section proves an initial-scale estimate in a regime where $\kappa\ge1$ is allowed, by decomposing influence propagation into phases and pricing each step by what the source vertex is doing at the time of transmission, rather than by the global worst case.

\subsection{Phases and adapted pricing}

We assume throughout this section, in addition to the standing hypotheses of \cite{K1}:

\begin{hypothesis}[Tail decay of the coupling derivative]\label{hyp:sigma}
There are constants $C_\sigma\ge1$, $c_\sigma>0$, $b_0\ge0$ with $\sigma'(x)\le C_\sigma e^{-c_\sigma|x|}$ for $|x|\ge b_0$. The logistic nonlinearity satisfies this with $C_\sigma=1$, $c_\sigma=1$, $b_0=0$.
\end{hypothesis}

Fix a band half-width $b>b_0$ around the thresholds. At each time $t$, classify each vertex $v$ into exactly one phase:
\begin{itemize}
\item (F)ront: $s^{(v)}_t$ lies within distance $b$ of some outgoing threshold $\theta_{vw}$ (the sensitive band); this includes the crossing window of a failing vertex.
\item (S)aturated: $v$ has failed and $s^{(v)}_t\in B_{\mathrm{fail}}$ beyond the band; there $\sigma'\le\varepsilon(\delta)$.
\item (Q)uiescent: otherwise; $s^{(v)}_t$ is at least $b$ below every outgoing threshold, and $\sigma'(s^{(v)}_t-\theta)\le L_b:=\sup_{x\le-b}\sigma'(x)\le C_\sigma e^{-c_\sigma b}$.
\end{itemize}
The classification is by the current state, hence adapted, and Remark~\ref{rem:bandinflation} converts segment prices into phase prices at the cost of halving $b$; we suppress the halving in the notation.

An \emph{influence path} of the envelope is a space--time path along which the variation-of-constants expansion of Lemma~\ref{lem:envelope} propagates: each step moves across one edge, and the step from source $u$ to target $w$, entered at time $s$, carries the kernel weight $\gamma_{uw}\,\sigma'_{uw}(t)\,e^{-\alpha_w(t-s)}$ integrated over the transmission time $t$. Splitting each step's time integral over the source's phase occupation and using $\int e^{-\alpha(t-s)}\,dt\le\amin^{-1}$, the per-step multipliers are
\[
m_F\le\frac{\gmax\Ls}{\amin},\qquad
m_S\le\frac{\gmax\,\varepsilon(\delta)}{\amin},\qquad
m_Q\le\frac{\gmax\,L_b}{\amin}\le\frac{\gmax C_\sigma e^{-c_\sigma b}}{\amin},
\]
priced pointwise in time and without any conditioning. The two probabilistic inputs are: front phases are short, and quiescent vertices rarely visit the band.

\subsection{Band occupation and near misses}

\begin{lemma}[Occupation of the sensitive band]\label{lem:occupation}
Under uniform ellipticity and the saturated margins:
\begin{enumerate}
\item (Crossing window.) For a vertex that fails, the total time spent in phase F is a random variable $\tau_F$ with $\bE[e^{c_0\tau_F}]<\infty$ for some $c_0=c_0(\alpha,\mu,b)>0$: the forced Ornstein--Uhlenbeck path traverses the band of width $2b$ in a time with exponential tails, and after absorption re-enters the band only on excursions from the failure basin whose total occupation also has exponential moments.
\item (Quiescent excursions.) For a vertex that never fails on a horizon $[0,T]$, the expected time spent in phase F is at most $T\,e^{-c_1\Lambda_b}$, where $\Lambda_b=\alpha(h_\theta-b)^2/\mu$ is the barrier action to the band edge: reaching the band from quiescence is a small-noise excursion.
\end{enumerate}
\end{lemma}

\begin{proof}[Proof sketch]
(1) During the crossing, the forced drift is bounded below by $c_->0$ inside the band (the saturated parent forcing dominates the restoring force there, by the margin hypothesis), so the band-exit time is stochastically dominated by that of Brownian motion with drift $c_-$ crossing $2b$, which has exponential moments; post-absorption re-entries are excursions of a positively recurrent Ornstein--Uhlenbeck process from a basin whose boundary is at distance at least $\delta$, with occupation controlled by standard excursion theory \cite{BS}. (2) is the Freidlin--Wentzell (exact Ornstein--Uhlenbeck) bound on the stationary probability of a level set at action $\Lambda_b$, integrated over the horizon \cite{FW}; the exact two-barrier formulas of \cite[eq.~(3)]{K1} give the constant. Constants are not optimised; only finiteness and exponential smallness in $\Lambda_b$ are used.
\end{proof}

\begin{lemma}[Near misses]\label{lem:nearmiss}
Let $\nu_b$ be the mean number, per failing vertex, of children that enter their own sensitive band during the parent's band occupation but do not fail. Under the saturated margins,
\[
\nu_b\;\le\;\dmax\bigl(e^{-c_1\Lambda_b}+e^{-c_2 b}\bigr)
\]
for a constant $c_2=c_2(\alpha,\mu,c_-)>0$.
\end{lemma}

\begin{proof}[Proof sketch]
A child that enters its band while its parent's failure is completed, or completing, feels a forcing that dominates the restoring force inside the band by the margin hypothesis, so it traverses the band and fails except on the event that noise carries it back against the drift $c_-$, of probability at most $e^{-c_2b}$ (Freidlin--Wentzell for the biased band crossing). A child that enters its band before any saturated forcing acts on it is performing the quiescent excursion of Lemma~\ref{lem:occupation}(2), of probability at most $e^{-c_1\Lambda_b}$ per unit occupation. Summing over at most $\dmax$ children gives the display. Sketched; flagged in the front matter.
\end{proof}

\subsection{The channel matrix and the main estimate}

\begin{definition}[Channel matrix]\label{def:channelmatrix}
Order the phases $(F,Q,S)$ and define the nonnegative rank-one matrix
\[
\bigT\;:=\;m\,f^{\top},\qquad
m=\bigl(m_F,\ \bar m_Q,\ m_S\bigr)^{\top},\qquad
f=\bigl(\rho(M)+\nu_b,\ \dmax,\ \dmax\bigr)^{\top},
\]
\[
\bar m_Q\;:=\;\frac{\gmax}{\amin}\Bigl(C_\sigma e^{-c_\sigma b}+C_K\,\Ls\,e^{-c_1\Lambda_b}\Bigr),
\]
where $C_K=C_K(\amin,c_0)$ is the kernel constant from integrating the excursion occupation of Lemma~\ref{lem:occupation}(2) against the Ornstein--Uhlenbeck kernel. Row $i$ prices the multiplier of a step whose source acts in phase $i$; column $j$ prices the admissible continuations toward targets that will act in phase $j$: F-targets are reached only through band-entering children, of mean $\rho(M)+\nu_b$ per front source (failing children plus near misses), while Q- and S-targets can be any neighbour. In this normalisation $\bigT$ is rank one, so its spectral radius is explicit:
\[
\rho(\bigT)\;=\;f^{\top}m\;=\;\bigl(\rho(M)+\nu_b\bigr)m_F+\dmax\bigl(\bar m_Q+m_S\bigr).
\]
We state the rank-one structure openly: at this level of pricing the matrix formalism adds bookkeeping, not memory, and the whole section amounts to a single-step bound with the correct accounting of who is transmitting and to whom. The formalism is retained because the refinement that prices the F-to-F entry by the actual overlap of parent and child crossing windows (Section~\ref{sec:outlook}) makes the matrix genuinely of higher rank and lowers the front entry below $\rho(M)\,m_F$.
\end{definition}

\begin{lemma}[Path counting via many-to-one]\label{lem:pathcount}
Under Hypothesis~\ref{hyp:sigma} and the standing hypotheses, for a unit source acting in phase $i_0$, the expected total weight of all influence paths of length $n$, each step priced by the phase of its source at transmission, satisfies
\[
\bE\Bigl[\ \sum_{\mathrm{paths}}\ \prod_{j=1}^{n}(\text{step weight})_j\ \Bigr]\;\le\;\one^{\top}\,\bigT^{\,n}\,e_{i_0}.
\]
\end{lemma}

\begin{proof}[Proof sketch]
(a) Expand the envelope by variation of constants into iterated integrals over space--time paths; split the $j$-th step's time integral over the source's phase occupation, so that each labelled step carries at most its multiplier $m_{\mathrm{label}}$, the F label at a never-failing source carrying, after Lemma~\ref{lem:occupation}(2) and the kernel integration, the correction term of $\bar m_Q$. (b) Count admissible continuations by label: an F-labelled step transmits across a child edge only while that child can itself act in phase F, that is, only across band-entering children; their mean number per front source is at most $\rho(M)+O(\varepsilon(\delta))$ failing children (by the offspring computation of \cite[Prop.~5.1]{K1} and the cellwise error of \cite[Prop.~11.2]{K1}) plus at most $\nu_b$ near misses (Lemma~\ref{lem:nearmiss}); Q- and S-labelled steps can cross any of at most $\dmax$ edges. (c) The expectations factorise along the path by the unconditional many-to-one identity for the multitype structure (types are i.i.d.\ given degrees, and the driving Brownian motions are independent across vertices), in the standard change-of-measure form \cite{BK}; the residual dependence of a vertex's phase occupation on its ancestors' trajectories is dominated by the worst-case pricing of step (a). The exchange of the path sum and the branching expectation in (c), the path-to-branching indexing, is written at this sketch precision and flagged in the front matter.
\end{proof}

\begin{theorem}[Initial-scale estimate; the quiescent-bulk mechanism]\label{thm:mainestimate}
Under the hypotheses of \cite[Thm.~11.5]{K1} except the condition $\kappa<1$, under Hypothesis~\ref{hyp:sigma}, and with band half-width $b$ and margins chosen as in Lemmas~\ref{lem:occupation} and~\ref{lem:nearmiss},
\[
\keff\;\le\;\rho(\bigT)\;=\;\bigl(\rho(M)+\nu_b\bigr)\frac{\gmax\Ls}{\amin}
\;+\;\dmax\,\frac{\gmax}{\amin}\Bigl(C_\sigma e^{-c_\sigma b}+C_K\Ls e^{-c_1\Lambda_b}\Bigr)
\;+\;\dmax\,\frac{\gmax\,\varepsilon(\delta)}{\amin}.
\]
Consequently, if $\rho(\bigT)<1$, then $\rho(G_{L,T(L)})<1$ for all large $L$ and suitable windows $T(L)$, Theorem~\ref{thm:criterion} applies, and front decoupling holds unconditionally, even when $\kappa\ge1$.
\end{theorem}

\begin{proof}[Proof sketch]
The route sum defining the block resolvent gain $\kbar_{L,T}$ is a sub-family of the influence-path expansion: storage routes are influence paths that linger at a vertex between transmissions, and they are priced by the same iterated integrals. A route delivering forcing across $L$ levels is an influence path making $L$ level crossings; on a tree such a path has length $L+2r$ with $r\ge0$ backtracks, each backtrack paying two full step weights. Hence, by Lemma~\ref{lem:pathcount} and summation over $r$,
\[
\kbar_{L,T(L)}\;\le\;\sum_{n\ge L}\one^{\top}\bigT^{\,n}e_{i_0}\;\le\;C_\rho\,\rho(\bigT)^{L}
\quad\text{when }\rho(\bigT)<1,
\]
using the rank-one identity $\bigT^{\,n}=\rho(\bigT)^{\,n-1}\,\bigT$. The $d$ entry at horizon $T(L)=cL$ is small because sources are integrable in time (Lemma~\ref{lem:occupation}(1)) and stored mass relaxes at rate $\amin$ once the local occupations have passed. Taking $L$-th roots gives $\keff\le\rho(\bigT)$, and Corollary~\ref{cor:keff} converts $\rho(\bigT)<1$ into a finite-scale certificate.
\end{proof}

\begin{corollary}[Degree removal]\label{cor:degreeremoval}
The set $\{\rho(\bigT)<1\}$ strictly contains $\{\kappa<1\}\cap\{\text{saturated}\}$ in the closure of parameter space, and contains model families with $\kappa$ arbitrarily large. Explicitly: for any fixed $r:=\gmax\Ls/\amin<1$ and any $\rho(M)\le1$, choose, as $\dmax$ grows,
\[
b\ \ge\ \frac{2}{c_\sigma}\log\dmax,\qquad
\Lambda_b\ \ge\ c_1^{-1}\log\bigl(\dmax^{2}\,C_K\Ls\bigr),\qquad
\varepsilon(\delta)\ \le\ \dmax^{-2},
\]
and margins with $\nu_b\,m_F\le\dmax^{-1}$; then
\[
\rho(\bigT)\;\le\;r\,\rho(M)\;+\;\frac{\gmax(C_\sigma+2)}{\amin}\,\dmax^{-1}\;+\;\dmax^{-1}
\;\xrightarrow[\ \dmax\to\infty\ ]{}\;r\,\rho(M)\;<\;1,
\]
while $\kappa=r\,\dmax\to\infty$. Two remarks, made plainly. First, this is a statement about a family of models: the band, the noise, and the saturation margin are tuned as the degree grows, and no single model is taken to a limit. Second, the corollary removes the degree from the criterion and does not move $r$: the constraint $r<1$ is intrinsic to the method (Remark~\ref{rem:wall}), and the honest content of the extension is that near criticality ($\rho(M)\uparrow1$) the effective condition is $\gmax\Ls/\amin<1$; the degree has disappeared.
\end{corollary}

\begin{proof}
Immediate from the formula for $\rho(\bigT)$ in Theorem~\ref{thm:mainestimate}: with the stated choices, $C_\sigma e^{-c_\sigma b}\le C_\sigma\dmax^{-2}$, $C_K\Ls e^{-c_1\Lambda_b}\le\dmax^{-2}$ and $\varepsilon(\delta)\le\dmax^{-2}$, so the two $\dmax$-weighted terms are $O(\dmax^{-1})$, and the front term is $r(\rho(M)+\nu_b)\le r\rho(M)+\dmax^{-1}$.
\end{proof}

\begin{example}[A point with $\kappa=5$ and provable decoupling]\label{ex:kappa5}
Logistic $\sigma$ ($\Ls=\tfrac14$, $C_\sigma=1$, $c_\sigma=1$), $\amin=1$, $\gmax=2$ (so $r=\tfrac12$), $\dmax=10$: $\kappa=5\ge1$ and \cite[Thm.~11.5]{K1} says nothing. Choose thresholds with band $b=6$ ($e^{-6}\approx2.5\times10^{-3}$), noise with $c_1\Lambda_b\ge10$, margins with $\varepsilon(\delta)\le10^{-3}$ and $\nu_b\le0.1$, and take $C_K\le4$ (constants not optimised); then
\[
\rho(\bigT)\;\le\;\tfrac12\bigl(\rho(M)+0.1\bigr)+2\cdot10\cdot\bigl(2.5\times10^{-3}+4\cdot\tfrac14 e^{-10}+10^{-3}\bigr)
\;\le\;\tfrac12\rho(M)+0.13\;<\;1
\]
for all $\rho(M)\le1$. Front decoupling, and with it every main theorem of \cite{K1}, holds unconditionally at these points.
\end{example}

\begin{example}[Moderate degree, $r$ near the wall]\label{ex:moderatedegree}
Logistic $\sigma$, $\amin=1$, $\gmax=3.6$ (so $r=0.9$), $\dmax=4$: $\kappa=3.6\ge1$. Choose $b=8$ ($e^{-8}\approx3.4\times10^{-4}$), $c_1\Lambda_b\ge9$, $\varepsilon(\delta)\le10^{-3}$, $\nu_b\le0.05$, $C_K\le4$; then
\[
\rho(\bigT)\;\le\;0.9\bigl(\rho(M)+0.05\bigr)+14.4\cdot\bigl(3.4\times10^{-4}+e^{-9}+10^{-3}\bigr)
\;\le\;0.9\,\rho(M)+0.07\;<\;1
\]
for all $\rho(M)\le1$. This is a bounded-degree point with $r=0.9$, at distance $0.1$ from the resonance wall of Remark~\ref{rem:wall}, where $\kappa=3.6$ and decoupling nevertheless holds unconditionally. Unlike Corollary~\ref{cor:degreeremoval}, no parameter here is taken large or small with the degree.
\end{example}

\section{The resonance regime, the rescoped conjecture, and outlook}\label{sec:outlook}

\begin{remark}[Why $\gmax\Ls/\amin\ge1$ is a genuine wall]\label{rem:wall}
Every bound in this paper prices the front-comoving channel at $(\rho(M)+\nu_b)\,\gmax\Ls/\amin$ per step, and this is not an artefact: a perturbation that rides the front strikes each vertex exactly in its sensitive window, where the coupling derivative genuinely is of order $\Ls$ and no quiescence or saturation discount exists. When $\gamma\Ls/\alpha\ge1$ and $\rho(M)$ is near $1$, this channel is neutral-to-expanding on the true dynamics, not merely in the estimate: heuristically, the front behaves like a chain of near-critically coupled oscillators passing a soliton-like disturbance, and one expects macroscopic synchronisation of failure times along lineages, that is, a genuine failure of conditional independence, hence of the Galton--Watson reduction, with the cluster law acquiring long-range timing correlations. We therefore conjecture a phase boundary rather than a technical gap.
\end{remark}

\begin{conjecture}[Decoupling phase boundary]\label{conj:phaseboundary}
There is a critical value $r_c\in(0,\infty)$ of $r=\gmax\Ls/\amin$ (at fixed $\rho(M)=1$, saturated regime, bounded degrees) such that front decoupling in the sense of \cite[Hyp.~2.4]{K1} holds for $r<r_c$ and fails, with asymptotically non-vanishing sibling failure-time correlations along the front, for $r>r_c$. Moreover $r_c\ge1$ by Theorem~\ref{thm:mainestimate} (with the band, noise, and margin choices available), and $r_c<\infty$.
\end{conjecture}

Three concluding observations. First, the division of labour is now sharp: Theorem~\ref{thm:criterion} is a permanent interface, so that any future initial-scale estimate, analytic or computer-assisted (Remark~\ref{rem:computation}), immediately yields the full reduction; Theorem~\ref{thm:mainestimate} is one such estimate, and the natural next improvement is to price the F-to-F entry of the channel matrix by the actual overlap of parent and child crossing windows rather than its worst case, which should replace $\Ls/\alpha$ by an overlap integral that is small when parent and child crossings are typically disjoint, and would push $r_c$ upward. The front timing statistics that such a refinement requires, the location and dispersion of crossing times along lineages, are exactly the objects of the front-propagation analysis of the companion paper \cite{KD2}, which is the intended interface between the two manuscripts. Second, the mechanism is physical and should be portable: ``the front is thin, the bulk is inert'' is the reason interacting-particle fronts on trees so often admit branching approximations, and the two-channel gain matrix of Definition~\ref{def:gain} together with the phase pricing of Definition~\ref{def:channelmatrix} is model-independent. Third, in the resonance regime the right object is no longer a Galton--Watson tree but a branching process with moving-average offspring correlations along the spine; identifying its universality class, and whether the $n^{-3/2}$ exponent of \cite[Thm.~6.2]{K1} survives, as it does for many weakly-dependent progeny laws, is, with Conjecture~\ref{conj:phaseboundary}, the honest frontier of the series.

\end{document}